\documentclass[leqno,a4paper,10pt]{article}

\usepackage[utf8x]{inputenc}
\usepackage{amsfonts}
\usepackage{amsmath}\usepackage{amssymb}
\usepackage{amsthm}
\usepackage{mathrsfs}
\usepackage{latexsym}
\usepackage{graphicx}
\usepackage{bm}
\usepackage{inputenc}
\usepackage[shortlabels]{enumitem}
\usepackage{hyperref}
\usepackage{graphicx}

\swapnumbers

\newtheorem{Lemma}{Lemma}[section]

\newtheorem{Theorem}[Lemma]{Theorem}

\theoremstyle{definition}

\newtheorem{Definition}[Lemma]{Definition}

\theoremstyle{remark}
\newtheorem*{Proof}{Proof}
\newtheorem{Remark}[Lemma]{Remark}
\newtheorem{Example}[Lemma]{Example}
\newtheoremstyle{citing}
{3pt}
{3pt}
{\itshape}
{}
{\bfseries}
{.}
{.5em}
{\thmnote{#3}}
\theoremstyle{citing}

\newtheoremstyle{proof*}
{3pt}
{3pt}
{\rmfamily}
{}
{ \itshape}
{.}
{.5em}
{\thmnote{#3}}
\theoremstyle{proof*}
\newtheorem*{proof*}{}

\DeclareMathOperator{\restrict}{\llcorner}

\DeclareMathOperator{\Clos}{Clos}  
\DeclareMathOperator{\Tan}{Tan}     
\DeclareMathOperator{\spt}{spt}     
\DeclareMathOperator{\im}{im}       
\DeclareMathOperator{\Lip}{Lip}     
\DeclareMathOperator{\dmn}{dmn}     
\DeclareMathOperator{\Nor}{Nor}
\DeclareMathOperator{\Hom}{Hom}     
\DeclareMathOperator{\Der}{D}       
       
\DeclareMathOperator{\ap}{ap}

\DeclareMathOperator{\reach}{reach}
\DeclareMathOperator{\trace}{trace}

\DeclareMathOperator{\Dual}{Dual}

\DeclareMathOperator{\Dis}{Dis}
\DeclareMathOperator{\discr}{discr}

\newcommand{\Real}[1]{ \mathbf{R}^{#1}}
\newcommand{\Haus}[1]{ \mathscr{H}^{#1} }
\newcommand{\Leb}[1]{ \mathscr{L}^{#1} }
\newcommand{\rect}[1]{(\mathscr{H}^{#1},#1)}

\newcommand{\Hdensity}[3]{\bm{\Theta}^{#1}(\mathscr{H}^{#1}\restrict \,#2,#3 )}

\newcommand{\infHdensity}[3]{\bm{\Theta}_{*}^{#1}(\mathscr{H}^{#1}\restrict \,#2,#3 )}

\title{Second order rectifiability of varifolds of bounded mean curvature}
\author{Mario Santilli}

\begin{document}
	\maketitle
	\begin{abstract}
	We prove that the support of an $ m $ dimensional rectifiable varifold with a uniform lower bound on the density and bounded generalized mean curvature can be covered $ \Haus{m} $ almost everywhere by a countable union of $m$ dimensional submanifolds of class $ \mathcal{C}^{2} $. We obtain this result using the notion of curvature of arbitrary closed sets originally developed in stochastic geometry and extending to our geometric setting techniques developed by Trudinger in the theory of viscosity solutions of PDE's.
	\end{abstract}

\paragraph{\small MSC-classes 2010.}{\small 49Q15, 53C65, 35D40, 35J60.}
\paragraph{\small Keywords.}{\small varifold, second order rectifiability, normal bundle, bounded mean curvature.}
	
	\section{Introduction}

The concept of varifold goes back to the work of Almgren in the 60's and, since then, has played a central role in Geometric Measure Theory and in its applications. The definition is simple: an $ m $-dimensional varifold $ V $ in an open subset $ \Omega $ of $ \Real{n} $ is a Radon measure over $ \Omega \times \mathbf{G}(n,m) $, where $ \mathbf{G}(n,m) $ is the Grassmann manifold of all $ m $ dimensional subspaces of $ \Real{n} $. Given such a $ V $, we define (1) \emph{the weight measure} $ \|V\| $ of $ V $\footnote{$ \|V\| $ is the Radon measure over $ \Omega $ such that $ \|V\|(U)=V(U \times \mathbf{G}(n,m)) $ for each open subset $ U $ of $ \Omega $.}, (2) the vector-valued distribution \mbox{$ \delta V : \mathcal{C}^{\infty}_{c}(\Omega, \Real{n}) \rightarrow \Real{} $} called \emph{(isotropic) first variation}\footnote{$ \delta V(g)  = \int \Der g \bullet S \, dV(x,S) $ for every $ g \in  \mathcal{C}^{\infty}_{c}(\Omega) $, that is the initial rate of change of the total mass of the smooth deformation of $ V $ with initial velocity given by $ g $.} and (3) the total variation\footnote{$ \| \delta V \| $ is the largest Borel regular measure over $ \Omega $ such that for each open set $ U \subseteq \Omega $ the number $	\| \delta V \|(U)$ equals $ \sup \{ \delta V(g) : g \in \mathcal{C}_{c}^{\infty}(U, \Real{n}), \; | g | \leq 1  \} $.} $ \|\delta V \| $ of $ \delta V $. If the $ m $ dimensional upper density $\bm{\Theta}^{\ast m}(\|V\| ,x)  $ is positive at $ \| V \| $ a.e.\ $ x \in \Omega $ and if $ \| \delta V \| $ is a Radon measure over $ \Omega $, then the celebrated rectifiability theorem of Allard \cite[5.5, 2.8(5)]{MR0307015} asserts that the set $ \{x: 0 < \bm{\Theta}^{\ast m}(\|V\|, x) < \infty \} $ can be $ \Haus{m} $ almost covered by the union of a countable collection of $ m $ dimensional submanifolds of class $ \mathcal{C}^{1} $ of $ \Real{n} $ and $ \|V\| = \Haus{m} \restrict \bm{\Theta}^{m}(\|V\|, \cdot) $. See also \cite{MR3794529} for a recent extension of Allard's rectifiability result to the anisotropic case.

The regularity theorems of Allard and Duggan, \cite[8]{MR0307015} and \cite[Theorem 2.1]{MR844169}, allows to conclude that if $ 2 \leq m < p < \infty $, $ \theta > 0 $, \mbox{$ \alpha < \infty $ and} $ V $ is an $ m $ dimensional varifold such that $ \bm{\Theta}^{m}(\|V\|,x) \geq \theta $ for $ \|V\| $ almost every $ x \in  \Omega $ and such that $ \delta V(g) \leq \alpha (\int |g|^{p/(p-1)}\, d\|V\|)^{(p-1)/p} $ for every $ g \in \mathcal{C}_{c}^{\infty}(\Omega, \Real{n}) $\footnote{the generalized mean curvature vector $ \mathbf{h}(V, \cdot) $ of $ V $ defined in \cite[4.3]{MR0307015} belongs to $ L^{p}(\|V\|, \Real{n}) $.}, then a dense open subset of $ \spt \|V\| $ is an $ m $ dimensional submanifold $ M $ of class $ W^{2, p} $. If we additionally assume that there exists $ \beta < 2 $ such that $ \bm{\Theta}^{m}(\|V\|, x) \leq \beta $ for $ \|V\| $ almost every $ x \in \Omega $, then the conclusion can be strengthened to $ \|V\|(\Omega \sim M) =0 $. However, one may construct \emph{integral} varifolds $ V $ of \emph{higher multiplicity} such that $ \|\delta V\| \leq \alpha \|V\| $ (i.e.\ $ \mathbf{h}(V, \cdot) \in L^{\infty}(\|V\|, \Real{n}) $) and $ V $ cannot be locally represented as a graph of multiple-valued function around each point of a set of positive $ \|V\| $ measure, see \cite[8.1(2)]{MR0307015} and \cite[6.1]{MR485012}. It follows that, in the case of higher multiplicity, the structure around almost every point of a varifold cannot be studied using classical regularity theory, even under the rather strong assumption $ \|\delta V\| \leq \alpha \|V\| $ (however, in the case $ \delta V =0 $, it is an open question if classical regularity holds almost everywhere). 

On the other hand it is reasonable to presume that an integrable mean curvature should entail a certain amount of regularity around almost every point and this regularity has been effectively discovered in recent years \emph{in the case of integral\footnote{The density function is integer-valued.} varifolds}.  In particular, the following results are currently known: (1) rectifiability of class $ \mathcal{C}^{2} $ has been completely solved in \cite[Theorem 1]{MR3023856} (see also \cite[5.1]{MR2064971}-\cite[3.1]{MR2472179} for the first positive result ever obtained in this direction), (2) tilt excess decay rates has been systematically clarified in most of the cases in \cite{MR485012}, \cite{MR2064971}, \cite{MR2898736}, \cite{MR3023856} and \cite{MR3625810}, and (3) the equivalence of quadratic decay rates and rectifiability of class $ \mathcal{C}^{2} $ has been proved in \cite[3.1]{MR2472179}. In contrast, \emph{for general rectifiable varifolds (i.e.\ the density function is real-valued), up to now, none of the aforementioned results is known} (not even in the stationary case $ \delta V =0 $). One problem to extend them to this more setting is that in the integral case they rely on the theory of $Q$-valued functions developed by Almgren and on a blow-up procedure, which has been originally developed by Brakke in \cite[5.6]{MR485012}. How to extend these techniques to non-integral varifolds is currently unclear.

In this paper, following a completely different approach, we prove rectifiability of class $ \mathcal{C}^{2} $ for varifolds with a uniform lower bound on the density and bounded generalized mean curvature, thus providing the first positive regularity results valid for almost every point of varifolds with real-valued densities with possible higher multiplicity.  Our main result reads as follows:

\begin{Theorem}\label{second order rect for varifolds}
 	Suppose $ 1 \leq m < n $ are integers, $\Omega \subseteq \Real{n} $ is an open set, $ V $ is an $ m $ dimensional varifold in $ \Omega $, $ S = \spt \|V\| $ and the following two conditions hold:
 	\begin{enumerate}
 		\item there exists $ 0 \leq h < \infty $ such that $ \| \delta V \| \leq h \|V\| $, 
 		\item there exists $ \theta > 0 $ such that $ \bm{\Theta}^{m}(\|V\|,x) \geq \theta $ for $ \|V\| $ a.e.\ $ x \in \Omega $.
 	\end{enumerate}
 
 Then $ S $ can be $\Haus{m}$ almost covered by a countable collection of $ m $ dimensional submanifolds of class $ 2 $ in $ \Real{n} $.
 \end{Theorem}

We explain now the strategy of the proof. The basic tools of our proof are taken from the theory of curvature for arbitrary closed sets, developed in \cite{MR534512}, \cite{MR2031455} and \cite{2017arXiv170801549S}. This theory is based on the definition for a closed subset $ A \subseteq \Real{n} $ of the \emph{generalized unit normal bundle of $ A $}:
\begin{equation*}
N(A) = (A \times \Real{n}) \cap \{ (a,u) : |u|=1, \;  \bm{\delta}_{A}(a+su)=s \; \textrm{for some $ s > 0 $}\}
\end{equation*}
(here $ \bm{\delta}_{A} $ is the distance function from $ A $), whose fiber at $ a $ is denoted by $N(A,a)$. Since $ N(A) $ is a countably $n-1$ rectifiable subset of $ \Real{n} \times \Real{n}$ (in the sense of \cite[3.2.14]{MR0257325}), one may use Coarea formula \cite[3.2.22]{MR0257325} with the projection-maps $ \mathbf{p} $ and $ \mathbf{q} $ (see section \ref{section 2} for notation) to generalize several integral formulas from smooth varieties to general closed sets (see \cite[Theorem 2.1]{MR2031455} and \cite[4.11(3), 5.4]{2017arXiv170801549S}). These formulas are expressed in terms of the \emph{generalized principal curvatures of $ A $} and \emph{the second fundamental form $ Q_{A} $} of $ A $; see section \ref{section 2} for more details. Of course, this theory alone is too general to produce useful results for our purpose. Therefore, in order to proceed, we need to understand how it specializes for the class of closed subsets that are supports of those varifolds considered in \ref{second order rect for varifolds}.  First, given an arbitrary closed set $ A \subseteq \Real{n} $, we introduce the following stratification of $ A $:
\begin{equation*}
A^{(m)} = A \cap  \{a : 0 < \Haus{n-m-1}(N(A,a)) < \infty  \} \quad \textrm{for $ m =0, \ldots , n $.}
\end{equation*}
The $ m $-th stratum $ A^{(m)} $ is the set of points where $ A $ can be touched by balls from $ n-m $ linearly independent directions. A crucial step for our result has been done in \cite{2017arXiv170309561M}, where it is proved that, for an arbitrary closed set $ A $, the $ m $-th stratum $ A^{(m)} $ can be covered by countably many $ m $ dimensional submanifolds of class $ 2$. Therefore the main point of the present paper is to show that if $ S $ is the support of a varifold as in \ref{second order rect for varifolds} then $ \Haus{m}(S \sim S^{(m)}) = 0 $. To prove it, we first introduce the following key definition.
 
\begin{Definition}\label{Lusin Property intro}
	Suppose $ A \subseteq \Real{n} $ is a closed set, $ \Omega \subseteq \Real{n} $ is an open set and $ 1 \leq m < n $ is an integer. We say that $ N(A) $ satisfies the \textit{$ m $ dimensional Lusin (N) condition in $ \Omega $} if and only if the following property holds:
	\begin{equation*}
	\Haus{n-1}(N(A)\cap \{(a,u) : a \in Z  \})=0
	\end{equation*}
	for every $Z \subseteq  A \cap \Omega$ with $\Haus{m}(A^{(m)} \cap Z) =0 $. 
\end{Definition}
Combining \cite[2.8]{MR3466806} with \cite[3.7]{2019arXiv190310379S} one concludes that the unit normal bundle of the support $ S $ of a varifold $ V $ as in \ref{second order rect for varifolds} satisfies the $ m $ dimensional Lusin (N) condition in $ \Omega $ and\footnote{More precisely, here we should consider the closure of $ S $ in $ \Real{n} $, since both the unit normal bundle and the second fundamental form are defined for closed subsets of $ \Real{n} $.} 
\begin{equation*}
	\trace Q_{S}(a,u) \leq h \quad \textrm{for $ \Haus{n-1} $ a.e.\ $ (a,u) \in N(S) $.}
\end{equation*}
This is essentially everything we need to known from varifold's theory and most of the results of this paper can actually be obtained for arbitrary closed sets whose normal bundle satisfies the Lusin (N) condition. The first important consequence of this assumption is the Coarea-type formula in \ref{area formula for the gauss map}. We use such a formula in the main result the paper (which is Lemma \ref{Around a point where the set is almost flat in one direction}) to extend one of the key results of the theory of viscosity solutions of elliptic PDE's, the Alexandrov-Bakelmann-Pucci (ABP) estimate (see \cite[Theorem 3.2]{MR1351007}), to our geometric setting. We do not explicitly write such a formula in the statement of our results, since the study of the ABP inequality in the context of varifolds (or, more generally, in the abstract setting of closed sets) would be beyond the scope of the present paper; however, the reader might recognize the resemblance in inequality \eqref{Around a point where the set is almost flat in one direction:11} of Lemma \ref{Around a point where the set is almost flat in one direction}. The validity of the ABP inequality is the central point to obtain the  criterion for rectifiability of class $ \mathcal{C}^{2} $ in \ref{final result about second order rectifiability}, whence, as one can easily see from what has been pointed out above, Theorem \ref{second order rect for varifolds} follows as a special case. The proof of Lemma 9 and its main consequence Theorem \ref{final result about second order rectifiability} are built upon a careful generalization of the argument employed by Trudinger in \cite[Theorem 1]{MR995142} to prove twice super-differentiability almost everywhere of a viscosity subsolution of an elliptc operator. A moment of reflection reveals that the conclusion of our Theorem \ref{final result about second order rectifiability}, $ \Haus{m}(A \sim A^{(m)}) =0 $, effectively corresponds to twice super-differentiability almost everywhere for $ A $ in an higher-codimensional and non-graphical setting.

We conclude noting that in this paper we do not use the full strength of Theorem \ref{final result about second order rectifiability}; in fact to prove Theorem \ref{second order rect for varifolds} it would have been enough to have $ f $ constant in \ref{final result about second order rectifiability}. However, we decide to state \ref{final result about second order rectifiability} with a much less restrictive hypothesis (and this hypothesis is maybe the optimal one) because it seems natural to think that this approach could also be useful to treat classes of varifolds with possibly unbounded mean curvature. However, verifying the Lusin (N) condition in these more general cases presents several additional non-trivial complications. It is our plan to investigate them in future works.

{\small\textbf{Acknowledgements.} The results of this paper were proved when the author was a Phd Student under the supervision of Prof.\ Ulrich Menne at the Max Planck Institute for Gravitational Physics. The author is grateful to his Phd advisor for his constant and supportive guidance throughout the preparation of this work.}

	\section{Notation and preliminary results}\label{section 2}

The open and closed balls of radius $ r $ and center $ a $ are respectively denoted by $ \mathbf{U}(a,r) $ and $ \mathbf{B}(a,r) $. The closure and the boundary in $ \Real{n} $ of a set $ A $ are denoted by $ \Clos A $ and $ \partial A $. The symbol $ \bullet $ denotes the standard inner product of $\Real{n}$. If $T$ is a linear subspace of $\Real{n}$, then $T_{\natural} : \Real{n} \rightarrow \Real{n}$ is the orthogonal projection onto $T$ and $T^{\perp} = \Real{n} \cap \{ v : v \bullet u =0 \; \textrm{for $u \in T$} \}$. If $X$ and $Y$ are sets and $ Z \subseteq X \times Y $ we set
\begin{equation*}
Z | S = Z \cap \{ (x,y) : x \in S  \} \quad \textrm{for $ S \subseteq X $.}
\end{equation*}
The maps $ \mathbf{p}, \mathbf{q} : \Real{n} \times \Real{n} \rightarrow \Real{n} $ are define by $\mathbf{p}(x,v)= x$ and $\mathbf{q}(x,v) = v$.

We adopt the language of \emph{symmetric algebra} to write in a compact form our formulas: if $ f : V \rightarrow W $ is a linear map between vector spaces, then there exists a unique linear map $\bigodot_{2}f : \bigodot_{2}V \rightarrow \bigodot_{2}W $, which is the restriction of the unique preserving algebra homeomorphism $ \bigodot_{\ast}f : \bigodot_{\ast}V \rightarrow \bigodot_{\ast}W $ onto $ \bigodot_{2}V $, see \cite[1.9]{MR0257325}.

\subsection{Curvatures of arbitrary closed sets}

\emph{The reference for this section is \cite{2017arXiv170801549S}.} 

Suppose $A$ is a closed subset of $ \Real{n} $. The \emph{distance function to $ A $} is denoted by $\bm{\delta}_{A} $. If $U$ is the set of all $x \in \Real{n}$ such that there exists a unique $a \in A$ with $|x-a| = \bm{\delta}_{A}(x)$, we define the \textit{nearest point projection onto~$A$} as the map $\bm{\xi}_{A}$ characterised by the requirement
	\begin{equation*}
	| x- \bm{\xi}_{A}(x)| = \bm{\delta}_{A}(x) \quad \textrm{for $x \in U$}.
	\end{equation*}
	Let $U(A) = \dmn \bm{\xi}_{A} \sim A$. The functions $ \bm{\nu}_{A} $ and $ \bm{\psi}_{A} $ are defined by
	\begin{equation*}
	\bm{\nu}_{A}(z) = \bm{\delta}_{A}(z)^{-1}(z -  \bm{\xi}_{A}(z)) \quad \textrm{and} \quad \bm{\psi}_{A}(z)= (\bm{\xi}_{A}(z), \bm{\nu}_{A}(z)),
	\end{equation*}
	whenever $ z \in U(A)$. We define the Borel function $ \rho(A, \cdot) $ setting
	\begin{equation*}
	\rho(A,x) = \sup \{t : \bm{\delta}_{A}(\bm{\xi}_{A}(x) + t (x-\bm{\xi}_{A}(x) ))=t \bm{\delta}_{A}(x)  \} \quad \textrm{for $ x \in U(A) $,}
	\end{equation*}
	and we say that $ x \in U(A) $ is a \emph{regular point of $ \bm{\xi}_{A} $} if and only if $ \bm{\xi}_{A}$ is approximately differentiable\footnote{See \cite[2.4, 2.6]{2017arXiv170801549S} for the definition of approximate differentiability.} at $ x $ with symmetric approximate differential and $ \ap \liminf_{y \to x} \rho(A,y) \geq \rho(A,x)>1$. The set of regular points of $ \bm{\xi}_{A} $ is denoted by $ R(A)$. It is proved in \cite[3.14]{2017arXiv170801549S} that $ \Leb{n}(\Real{n} \sim (A \cup R(A))) =0 $ and if $ x \in R(A) $ then $ \bm{\xi}_{A}(x) + t(x - \bm{\xi}_{A}(x)) \in R(A) $ for every $ 0 < t < \rho(A,x) $. Moreover, $ \bm{\psi}_{A}| \{ x: \bm{\delta}_{A}(x) = r, \; \rho(A,x) \geq \lambda \} $ is a bi-lipschitzian homeomorphism whenever $ r > 0 $ and $ \lambda > 1 $, see \cite[3.17(1)]{2017arXiv170801549S}. 
	
	Combining these two facts, we now briefly describe how a general notion of second fundamental form for arbitrary closed sets has been introduced in \cite[section 4]{2017arXiv170801549S}. This notion will be repeatedly used in the rest of this paper. First of all, we define \emph{the generalized unit normal bundle of $ A $} as
	\begin{equation*}
	N(A) = (A \times \mathbf{S}^{n-1}) \cap \{ (a,u) :  \bm{\delta}_{A}(a+su)=s \; \textrm{for some $ s > 0 $}\}, 
	\end{equation*}
	with $ N(A,a) = \{ v : (a,v) \in N(A)   \} $ for $ a \in A $. Since 
	\begin{equation*}
	N(A) = \bigcup_{r > 0}\bm{\psi}_{A} [\{ x: \bm{\delta}_{A}(x) = r, \; \rho(A,x) \geq \lambda \}] \quad \textrm{for every $ \lambda > 1 $,}
	\end{equation*}
	one uses the rectifiability properties of the distance sets $ \{x: \bm{\delta}_{A}(x)=r \} $ (see \cite[2.13]{2017arXiv170801549S}) to conclude that $ N(A) $ is a countably $n-1$ rectifiable subset of $ \Real{n}  \times \mathbf{S}^{n-1} $ in the sense of \cite[3.2.14]{MR0257325}. Then we introduce the following definition: if $ x \in R(A) $ then we say that $ \bm{\psi}_{A}(x) $ is \emph{a regular point of $ N(A) $}, and we denote the set of all regular points of $N(A)$ by $R(N(A))$. One may check (see \cite[4.5]{2017arXiv170801549S}) that $ \Haus{n-1}(N(A) \sim R(N(A)) =0 $. For every $ (a,u) \in R(N(A)) $, if $ x \in R(A) $ and $ \bm{\psi}_{A}(x)= (a,u) $, we define 
	\begin{equation*}
	T_{A}(a,u) = \im \ap \Der \bm{\xi}_{A}(x),
	\end{equation*}
	and we define a symmetric bilinear form $ Q(a,u): T_{A}(a,u) \times T_{A}(a,u) \rightarrow \Real{} $ which maps $ (\tau, \tau_{1}) \in T_{A}(a,u) \times T_{A}(a,u) $ into
	\begin{equation*}
	Q_{A}(a,u)(\tau, \tau_{1})= \tau \bullet \ap \Der \bm{\nu}_{A}(x)(\sigma_{1});
	\end{equation*}
here $ \sigma_{1} \in \Real{n} $ is any vector such that $ \ap \Der \bm{\xi}_{A}(x)(\sigma_{1}) = \tau_{1}$. This is  a well-posed definition, see \cite[4.6, 4.8]{2017arXiv170801549S}. We call $ Q_{A}(a,u) $ \textit{second fundamental form of $ A $ at $ a $ in the direction $ u $.} It is not difficult to check that if $ A $ is smooth submanifold, then $Q_{A}$ agrees with the classical notion of differential geometry. Moreover, if $(a,u) \in R(N(A))$ we define \emph{the principal curvatures of $ A $ at $(a,u) $} to be the numbers 
	\begin{equation*}
	\kappa_{A,1}(a,u) \leq \ldots \leq \kappa_{A,n-1}(a,u),
	\end{equation*}
	such that $\kappa_{A,m +1}(a,u) = \infty $, $\kappa_{A,1}(a,u), \ldots , \kappa_{A,m}(a,u)$ are the eigenvalues of $ Q_{A}(a,u)$ and $ m = \dim T_{A}(a,u)$.
	
	\subsection{The second-order rectifiable stratification}\label{section stratification}
\emph{The reference for this section is \cite{2017arXiv170309561M}.}

Suppose $ A \subseteq \Real{n} $ is a closed subset of $ \Real{n} $. For each $ a \in A $ we define (see \cite[4.1, 4.2]{2017arXiv170309561M}) the closed convex subset
	\begin{equation*}
	\Dis(A,a) = \{ v : |v| = \bm{\delta}_{A}(a+v)  \}
	\end{equation*}
	and we notice that $ N(A,a) = \{  v/|v|: 0 \neq v \in \Dis(A,a)  \} $. For every integer $ 0 \leq m \leq n $ we define \textit{the $m$-th stratum of $A$} by
	\begin{equation*}
	A^{(m)} = A \cap \{ a : \dim \Dis(A,a) = n-m \};
	\end{equation*}
	this is a Borel set which is countably $ m $ rectifiable in the sense of \cite[3.2.14]{MR0257325} and it can be $ \Haus{m} $ almost covered by the union of a countable family of $ m $ dimensional submanifolds of class $ \mathcal{C}^{2} $; see \cite[4.12]{2017arXiv170309561M}. This definition agrees with \cite[5.1]{2017arXiv170801549S} by \cite[4.4]{2017arXiv170309561M}. Moreover, on may use the classical Coarea formula for functions to infer that
	\begin{equation}\label{curvature of arbitrary closed sets 4: eq1} 
		A^{(m)} = A \cap  \{a : 0 < \Haus{n-m-1}(N(A,a)) < \infty  \} \quad \textrm{if $ m = 0, \ldots , n-1 $, }
	\end{equation}
	\begin{equation*}
		A^{(n)} = A \cap \{ a : N(A,a) = \varnothing   \}.
	\end{equation*}
	This stratification and its rectifiability properties will play a crucial role in our results. In fact, we achieve rectifiability of class $ \mathcal{C}^{2} $ for a varifold $ V $ as in \ref{second order rect for varifolds} proving that $ \Haus{m}(\spt \|V\| \sim (\spt \|V\|)^{(m)}) =0 $.
	 

\subsection{Curvature under diffeomorphic deformations}

In this section we prove an explicit formula for the second fundamental form $ Q_{F[A]} $ of a diffeomorphic deformation $F[A]$ of an arbitrary closed set $ A $, in terms of $ Q_{A} $. This formula appears to be new even in the smooth setting.

\begin{Lemma}\label{normal bundle and diffeomorphisms}
	Suppose $ A \subseteq \Real{n} $ is a closed set, $ F : \Real{n} \rightarrow \Real{n} $ is a diffeomorphism of class $ 2 $ onto $ \Real{n} $ and $ \nu_{F}: \Real{n}\times \mathbf{S}^{n-1} \rightarrow \Real{n}\times \mathbf{S}^{n-1} $ is given by
	\begin{equation*}
	\nu_{F}(a,u)=\left( F(a), \frac{  (\Der F(a)^{-1})^{*}(u)    }{ | (\Der F(a)^{-1})^{*}(u) |     }\right) \quad \textrm{whenever $ (a,u )\in \Real{n} \times \mathbf{S}^{n-1} $}.
	\end{equation*}
	
	Then $\nu_{F}$ is a diffeomorphism of class $1$ onto $\Real{n}\times \mathbf{S}^{n-1}$, $(\nu_{F})^{-1} = \nu_{F^{-1}}$ and
	\begin{equation}\label{normal bundle and diffeomorphisms: eq1}
	 \nu_{F}\big(N(A)\big) = N\big(F(A)\big).
	\end{equation}
	In particular, $ F\big(A^{(m)}\big) = F(A)^{(m)} $ for $ m = 0 , \ldots , n $.
\end{Lemma}

\begin{proof}
	A direct computation shows that $ \nu_{F} $ is a diffeomorphism of class $ 1 $ onto $ \Real{n} \times \mathbf{S}^{n-1} $ with $ (\nu_{F})^{-1} = \nu_{F^{-1}} $.
	
	If $ (a,u) \in N(A) $ and $ r >0 $ such that $ \mathbf{U}(a+ru,r) \cap A = \varnothing $, we let
	\begin{equation*}
	v=(\Der F(a)^{-1})^{*}(u), \quad 	W = F\big(\mathbf{U}(a+ru,r)\big), \quad S = \partial W.
	\end{equation*}
	Since $ S = F\big(\partial \mathbf{U}(a+ru,r)\big) $, by \cite[3.1.21]{MR0257325} we conclude that
	\begin{equation*}
	\Der F(a)\big(\Tan(\partial \mathbf{U}(a+ru,r),a)\big)= \Tan(S, F(a)),
	\end{equation*}
	and, consequently, $ v \in \Nor(S,F(a)) $. If $ s = \reach (S,F(a)) $ (see \cite[4.1]{MR0110078}), then by \cite[4.11, 4.8(12)]{MR0110078} we conclude that $s >0 $, 
	\begin{equation*}
	\mathbf{U}(F(a) + s (v/|v|),s) \cap S  = \varnothing
	\end{equation*}
and we deduce that
	\begin{center}
		either $ \mathbf{U}(F(a) + s (v/|v|),s) \subseteq W $ or $ \mathbf{U}(F(a) + s (v/|v|),s) \subseteq \Real{n} \sim \Clos W $.
	\end{center}
	If $ \gamma(t) = F(a + tu) $ for $ t \in \Real{} $, noting that $ \dot{\gamma}(0) \bullet v = 1 $ and 
	\begin{equation*}
	\Der_{t}(|\gamma(t) - F(a)-s(v/|v|)|)(0) = -1/|v|,
	\end{equation*}
	we conclude that $ \gamma(t) \in \mathbf{U}(F(a) + s (v/|v|),s) $ for $ t > 0 $ sufficiently small,
	\begin{equation*}
	\mathbf{U}(F(a) + s (v/|v|),s) \subseteq W \quad \textrm{and} \quad \nu_{F}(a,u) \in N\big(F(A)\big).
	\end{equation*}
	Therefore $ \nu_{F}\big(N(A)\big)\subseteq N\big(F(A)\big) $ and replacing $ F $ by $F^{-1}$ and $A$ by $F(A)$ we conclude
	\begin{equation*}
	\nu_{F}\big( N(A)\big) =  N\big(F(A)\big).
	\end{equation*}
	Noting that for each $ a \in \Real{n} $ the function mapping $ u \in \mathbf{S}^{n-1} $ onto $ \mathbf{q}(\nu_{F}(a,u)) $ is a diffeomorphism onto $ \mathbf{S}^{n-1} $ the postscript follows from \eqref{curvature of arbitrary closed sets 4: eq1} and \eqref{normal bundle and diffeomorphisms: eq1}
\end{proof}

\begin{Theorem}\label{second fundamental form and diffeomorphisms}
	Suppose $A$ is a closed subset of $\Real{n}$ and $ F : \Real{n} \rightarrow \Real{n} $ is a diffeomorphism of class $ 2 $ onto $\Real{n}$. 
	
	Then (see \ref{normal bundle and diffeomorphisms}) $\Der F(a)\big(T_{A}(a,u)\big) = T_{F(A)}(\nu_{F}(a,u))$ and
	\begin{flalign*}
	& Q_{F(A)}(\nu_{F}(a,u))\circ \textstyle \bigodot_{2}(\Der F(a)| T_{A}(a,u))   \\
	& \quad = |(\Der F(a)^{-1})^{\ast}(u)|^{-1}Q_{A}(a,u) \\
	& \quad \quad + (\Der^{2}F(a)| \textstyle \bigodot_{2}T_{A}(a,u)) \bullet ((\Der F(a)^{-1})^{\ast}(u)/|(\Der F(a)^{-1})^{\ast}(u)|),
	\end{flalign*}
	for $\Haus{n-1}$ a.e.\ $(a,u)\in N(A)$.
\end{Theorem}
\begin{proof}
We define $ g : \Real{n} \times \Real{n} \rightarrow \Real{n} $ to be
	\begin{equation*}
	g(a,u) = (\Der F(a)^{-1})^{\ast}(u) \quad \textrm{for $(a,u) \in \Real{n} \times \Real{n} $}.
	\end{equation*}
To compute $ \Der g $, we first notice that 
\begin{equation*}
	g = [ \beta \circ (\alpha \times \bm{1}_{\Real{n}}) \circ (\iota \times  \bm{1}_{\Real{n}}) \circ (\Der F \times  \bm{1}_{\Real{n}})  ], 
\end{equation*}	
where $ \iota(T) = T^{-1} $ for every isomorphism $ T $ of $ \Real{n} $, $ \alpha(T) = T^{\ast} $ for $ T \in \Hom(\Real{n}, \Real{n}) $ and $ \beta(T,u)=T(u) $ for every $ (T,u) \in  \Hom(\Real{n}, \Real{n}) \times \Real{n} $. Then differentiating such a composition of maps one obtains that
\begin{flalign}\label{second fundamental form and diffeomorphisms eq1}
	\Der g(a,u)(\tau, \sigma) & = (\Der F(a)^{-1})^{\ast}(\sigma) \notag \\
	& \quad  +  [(\Der F(a)^{-1})^{\ast} \circ \Der(\Der F)(a)(\tau)^{\ast} \circ (\Der F(a)^{-1})^{\ast} ](u),
	\end{flalign}
for $(a,u) $ and $ (\tau, \sigma) $ in $ \Real{n} \times \Real{n}	$. Moreover, one can easily compute that
	\begin{equation}\label{second fundamental form and diffeomorphisms eq2}
	\Der \Big( \frac{g}{|g|}\Big)(a,u) = \frac{1}{|g(a,u)|}\Big[ \Der g(a,u) - \Big( \Der g(a,u) \bullet \frac{g(a,u)}{|g(a,u)|}  \Big) \frac{g(a,u)}{|g(a,u)|} \Big]
	\end{equation}
for $(a,u) \in \Real{n}	\times (\Real{n} \sim \{ 0 \}) $ and $ (\tau, \sigma) \in \Real{n} \times \Real{n} $. 

Let $\theta$ be $ \Haus{n-1}\restrict N(A) $ measurable and $ \Haus{n-1} \restrict N(A) $ almost positive function such that $ \theta \Haus{n-1} \restrict N(A) $ is a Radon measure. Noting \ref{normal bundle and diffeomorphisms}, we define
	\begin{equation*}
	\mu = (\theta \circ \nu_{F^{-1}}) \Haus{n-1} \restrict N(F[A])
	\end{equation*}
	and we apply \cite[B.1]{2017arXiv170801549S} with $ \gamma = \Lip(\nu_{F}|\nu_{F^{-1}}[K])$ to conclude that  
	\begin{equation*}
	\mu(K) \leq \gamma^{n-1}\int_{N(A)\cap \nu_{F^{-1}}[K]}\theta \, d\Haus{n-1} < \infty,
	\end{equation*}
	whenever $ K \subseteq \Real{n}\times \mathbf{S}^{n-1}$ is compact. Let $ \psi = \theta \Haus{n-1} \restrict N(A) $. Noting again \ref{normal bundle and diffeomorphisms}, one may use \cite[4.11(1)]{2017arXiv170801549S} and \cite[B.2]{2017arXiv170801549S} to see that for $ \Haus{n-1} $ a.e.\ $ (a,u) \in N(A) $ the approximate tangent cones $\Tan^{n-1}(\psi, (a,u))$ and $\Tan^{n-1}(\mu, \nu_{F}(a,u))$ are $n-1 $ dimensional planes in $ \Real{n} \times \Real{n} $ and
	\begin{equation*}
	\Der \nu_{F}(a,u)[\Tan^{n-1}(\psi, (a,u))] = \Tan^{n-1}(\mu, \nu_{F}(a,u)).
	\end{equation*}
	 Employing \cite[4.11(2)]{2017arXiv170801549S} one infers for $ \Haus{n-1} $ a.e.\ $ (a,u) \in N(A) $ that
	\begin{equation}\label{second fundamental form and diffeomorphisms eq3}
	\Der F(a)[T_{A}(a,u)] = T_{F[A]}(\nu_{F}(a,u)),
	\end{equation}
	\begin{flalign}\label{second fundamental form and diffeomorphisms eq4}
	& Q_{F[A]}(\nu_{F}(a,u))(\Der F(a)(\tau), \Der F(a)(\tau_{1})) \notag \\
	& \qquad  \qquad  = \Der F(a)(\tau) \bullet \Der (g/|g|)(a,u)(\tau_{1}, \sigma_{1})
	\end{flalign}
	whenever $\tau \in T_{A}(a,u)$, $\tau_{1} \in T_{A}(a,u)$ and $(\tau_{1},\sigma_{1}) \in \Tan^{n-1}(\psi, (a,u))$. Since 
	\begin{equation*}
		T_{A}(a,u) \subseteq \{v : v \bullet u =0  \}
	\end{equation*}
	for $ \Haus{n-1} $ a.e.\ $(a,u)\in N(A)$ by \cite[4.5, 4.8]{2017arXiv170801549S}, it follows from \eqref{second fundamental form and diffeomorphisms eq3} that
	\begin{equation}\label{second fundamental form and diffeomorphisms eq5}
\Der F(a)(\tau) \bullet (g/|g|)(a,u) = |g(a,u)|^{-1} u \bullet \tau =0
	\end{equation}
for $ \Haus{n-1} $ a.e.\ $ (a,u) \in N(A) $ and for every $ \tau \in T_{A}(a,u) $. Therefore combining \eqref{second fundamental form and diffeomorphisms eq4}, \eqref{second fundamental form and diffeomorphisms eq2}, \eqref{second fundamental form and diffeomorphisms eq5} and \eqref{second fundamental form and diffeomorphisms eq1} we obtain for $ \Haus{n-1} $ a.e.\ $ (a,u) \in N(A) $ that
\begin{flalign*}
& Q_{F[A]}(\nu_{F}(a,u))(\Der F(a)(\tau), \Der F(a)(\tau_{1})) \\
& \qquad = |g(a,u)|^{-1} \Der F(a)(\tau) \bullet \Der g(a,u)(\tau_{1}, \sigma_{1}) \\
& \qquad = |g(a,u)|^{-1}[(\tau \bullet \sigma_{1}) + \Der(\Der F)(a)(\tau_{1})(\tau) \bullet (\Der F(a)^{-1})^{\ast}(u)  ]
\end{flalign*}
for every $ \tau, \tau_{1} \in T_{A}(a,u)$ and $(\tau_{1}, \sigma_{1}) \in \Tan^{n-1}(\psi, (a,u)) $. This is our conclusion by \cite[4.11(2)]{2017arXiv170801549S}.
\end{proof}	

\section{A sufficient condition for $\mathcal{C}^{2}$ rectifiability for closed sets}

This section is the main technical part of the paper. We work in the abstract setting of closed subsets whose generalized unit normal bundle satisfies the Lusin (N) condition. The main point here is to provide a general criterion for rectifiability of class $ \mathcal{C}^{2} $ (see Theorem \ref{final result about second order rectifiability}). Then, in the next section we verify that the support of a varifold as in Theorem \ref{second order rect for varifolds} satisfies the hypothesis of this criterion, thus obtaining the announced result for varifolds.

\begin{Definition}\label{Lusin Property}
	Suppose $ A \subseteq \Real{n} $ is a closed set, $ \Omega \subseteq \Real{n} $ is an open set and $ 1 \leq m < n $ is an integer. We say that $ N(A) $ satisfies the \textit{$ m $ dimensional Lusin (N) condition in $ \Omega $} if and only if
	\begin{equation*}
		\Haus{n-1}(N(A)|S)=0, \quad \textrm{whenever $S \subseteq  A \cap \Omega$ such that $\Haus{m}(A^{(m)} \cap S) =0 $.}
	\end{equation*}
	In case $ \Omega = \Real{n} $, we say that $N(A)$ satisfies the \textit{$ m $ dimensional Lusin (N) condition.}
\end{Definition}

We have introduced this terminology in analogy with the theory of functions: $ f : \Real{n} \rightarrow \Real{n} $ is said to satisfy the Lusin (N) condition if $ \Leb{n}(f(A)) =0 $ whenever $ \Leb{n}(A) =0 $, see \cite{MR1182504}. Actually, we can think $N(A)$ to be a set-valued function associating at each point $a$ the set $N(A,a) $. Therefore we can interpret the Lusin (N) condition given in \ref{Lusin Property} as a property of the graph of $N(A)$. 

\begin{Remark}\label{localizing the Lusin property}
	Suppose $ A $ is a closed subset of $\Real{n}$, $ \Omega $ is an open subset of $\Real{n}$ and $ C = \Clos(A \cap \Omega)$. Then one may easily check that
	\begin{equation*}
	N(A)|\Omega = N(C)| \Omega.
	\end{equation*}
	It follows that $ A^{(m)} \cap \Omega = C^{(m)} \cap \Omega $ for every $ m = 0 , \ldots , n $ by \eqref{curvature of arbitrary closed sets 4: eq1}, whence we deduce that \emph{if $N(A)$ satisfies the $m$ dimensional Lusin (N) condition in $ \Omega $, then $N(C)$ satisfies the $m$ dimensional Lusin (N) condition in $ \Omega $.} Moreover, 
	\begin{equation*}
	Q_{A}(\zeta) = Q_{C}(\zeta)
	\end{equation*}
	for $ \Haus{n-1}$ a.e.\ $ \zeta \in N(A)|\Omega$ by \cite[4.14]{2017arXiv170801549S}.
\end{Remark}

\begin{Remark}\label{Lusin condition and dimension of tangent space}
	If $N(A)$ satisfies the $ m $ dimensional Lusin (N) condition in $ \Omega $ then it follows from \cite[6.1]{2017arXiv170801549S} and \cite[4.12]{2017arXiv170309561M} that 
	\begin{equation*}
	\dim T_{A}(a,u) = m \quad \textrm{for $ \Haus{n-1} $ a.e.\ $ (a,u) \in N(A)| \Omega $.}
	\end{equation*}
\end{Remark}

\begin{Lemma}\label{Lusin property and diffeomorphisms}
	Suppose $ U \subseteq \Real{n} $ is open, $ A \subseteq \Real{n} $ is closed, $N(A)$ satisfies the $ m $ dimensional Lusin (N) condition in $U$ and $ F : \Real{n}\rightarrow \Real{n} $ is a diffeomorphism of class $ 2 $ onto $ \Real{n} $.

Then $ N\big(F(A)\big) $ satisfies the $ m $ dimensional Lusin (N) condition in $ F(U) $.
\end{Lemma}
\begin{proof}
Suppose $S \subseteq F(A)\cap F(U) $ such that $\Haus{m}\big(F(A)^{(m)} \cap S\big)=0$. Since $F^{-1}(S) \subseteq A \cap U$ and $ 0 = \Haus{m}\big(F^{-1}(S \cap F(A)^{(m)})\big)= \Haus{m}\big(F^{-1}(S) \cap A^{(m)}\big) $ by \ref{normal bundle and diffeomorphisms}, it follows by the Lusin (N) condition of $N(A)$ that
\begin{equation*}
 \Haus{n-1}\big(N(A)|F^{-1}(S)\big) =0.
\end{equation*}
Then \ref{normal bundle and diffeomorphisms} implies that 
\begin{equation*}
	\nu_{F}\big(N(A)|F^{-1}(S)\big) = N\big(F(A)\big)|S, \quad \Haus{n-1}\big(N(F(A))|S\big) =0.
\end{equation*}
\end{proof}

The preservation of the Lusin (N) condition under diffeomorphisms is a subtle point. In fact, the following example shows that if we had define the Lusin condition in \ref{Lusin Property} replacing $\Haus{n-1}(N(A)|S)=0 $ with the weaker property $ \Haus{n-1}(\mathbf{q}(N(A)|S)) =0 $, then the resulting condition would not be preserved under diffeomorphisms, as the following example shows for $ n = 3 $ and $ m = 2 $.

\begin{Example}\label{weaker Lusin condition}
Suppose $A = \Real{3} \cap \{ (x,y,z): z=|x| \}$ and $ F : \Real{3} \rightarrow \Real{3} $ is given by $ F(x,y,z)= (x,y,z+ 1-x^{2}-y^{2}) $ for $ (x,y,z) \in \Real{3} $. Then one readily verifies that $ \Haus{2}(\mathbf{q}(N(A))) =0 $. On the other hand,
\begin{equation*}
 A^{(1)} = \Real{3}\cap \{ (x,y,z) : x=z=0  \},
\end{equation*}
\begin{equation*}
	F(A)^{(1)} = F(A^{(1)}) = \{ (0,y,1-y^{2}) : y \in \Real{} \},
\end{equation*}
and the relative interior in $ \mathbf{S}^{2} $ of $\mathbf{q}[N(F(A))|F(A)^{(1)}] $ is non empty, as one may see by computing explicitly $F[A]$.
\end{Example}

One of the main consequences of the Lusin (N) condition is the following Coarea-type formula, whose proof is given in \cite[3.3]{2019arXiv190310379S}. 
\begin{Theorem}\label{area formula for the gauss map}
	Suppose $ 1 \leq m < n $ is an integer, $ \Omega \subseteq \Real{n} $ is open, $ A \subseteq \Real{n} $ is closed and $N(A)$ satisfies the $ m $ dimensional Lusin (N) condition in $ \Omega $.
	
Then for every $\Haus{n-1}$ measurable set $ B \subseteq N(A)|\Omega $,
	\begin{equation*}
		\int_{\mathbf{S}^{n-1}} \Haus{0}\{ a: (a,u)\in B  \}\,d\Haus{n-1}u  = \int_{A}\int_{B(z) } | \discr Q_{A}| d\Haus{n-m-1} d\Haus{m}z.
	\end{equation*}
\end{Theorem}

We need the following simple fact from linear algebra in the proof of the next result.
\begin{Lemma} \label{Linear algebra trace}
	Suppose $ V$ and $W $ are finite dimensional vector spaces with inner products such that $ \dim V = m $ and $\dim W = n $, $ f \in \Hom (V,W) $, $ 0 < t < \infty $ and $ b \in \textstyle \bigodot^{2}W $ such that $ b(w,w)\leq t|w|^{2} $ whenever $ w \in W $.
	
	Then
	\begin{equation*}
		 \| f \|^{2}\trace(b) + (1-n)t \|f \|^{2}\leq \trace\left( b \circ \textstyle \bigodot_{2}f \right) \leq m t\|f\|^{2}.
	\end{equation*}
\end{Lemma}
\begin{Proof}
By \cite[1.7.3]{MR0257325} we can choose an orthonormal basis $ v_{1}, \ldots , v_{m} $ of $V$ and an orthonormal basis $w_{1}, \ldots , w_{n}$ of $ W $ such that
\begin{equation*}
	(f^{\ast} \circ f)(v_{i}) \bullet v_{j} =0 \quad \textrm{and} \quad b(w_{i},w_{j})=0,
\end{equation*}
whenever $ i \neq j $. If we define $c(w,z)= t (w \bullet z) - b(w,z) $ whenever $ w,z \in W$, noting $\|f\| = \|f^{\ast}\|$ by \cite[1.7.6]{MR0257325}, we compute
	\begin{flalign*}
		& \textstyle \trace (c \circ \textstyle \bigodot_{2}f) = \sum_{i=1}^{m}\sum_{j=1}^{n} (f(v_{i})\bullet w_{j})^{2}c(w_{j},w_{j})\\
		& \quad \textstyle = \sum_{j=1}^{n} | f^{*}(w_{j})|^{2}c(w_{j},w_{j}) \leq \| f\|^{2}(nt-\trace b),\\
		& \trace( c \circ \textstyle \bigodot_{2}f ) = t \sum_{i=1}^{m}|f(v_{i})|^{2} - \trace( b \circ \textstyle \bigodot_{2}f ) \geq t \|f\|^{2} - \trace ( b \circ \textstyle \bigodot_{2}f ).
	\end{flalign*}
	Combining the two equations we get the left side. The right side is trivial.
\end{Proof}

\begin{Definition}
	If $ 0 < t < \infty $, $ a \in \Real{n}$ an $ T \in \mathbf{G}(n,n-1) $, we define
	\begin{equation*}
	C_{t}(T,a) = \Real{n} \cap \{ x : |T_{\natural}(x-a)| < t, \; |T^{\perp}_{\natural}(x-a)| < t \}.
	\end{equation*}
\end{Definition}

The criterion for second-order-differentiability in \ref{final result about second order rectifiability}, that is the central result of this section, can be deduced by standard arguments from the somewhat more subtle result in \ref{Around a point where the set is almost flat in one direction}.

\begin{Lemma}[Main Lemma]\label{Around a point where the set is almost flat in one direction}
	If $ 1 \leq m < n $ are integers, then there exist $ 0 < \delta < \infty $ and $ 0 < \sigma < \infty $ such that the following statement holds.
	
	If $ A \subseteq \Real{n} $ is a closed set, $a \in A$, $ 0 < r < \infty $, $T \in \mathbf{G}(n,n-1)$ and the following three conditions hold,
	\begin{enumerate}[(I)]
		\item $N(A)$ satisfies the $ m $ dimensional Lusin (N) condition in $C_{4r}(T,a) $,
		\item there exists $ v \in \mathbf{S}^{n-1} $ such that $ T_{\natural}(v) =0 $ and
		\begin{equation*}
		\sup \{ v \bullet (x-a): x \in \Clos(A \cap C_{4r}(T,a))  \} \leq r/16,
		\end{equation*} 
		\item\label{Around a point where the set is almost flat in one direction:hp3} there exists a nonnegative $\Haus{n-1}$ measurable function $f$ on $N(A)$ such that
		\begin{equation*}
		\trace Q_{A}(x,u) \leq f(x,u) \quad \textrm{for $\Haus{n-1}$ a.e.\ $(x,u) \in N(A)|C_{2r}(T,a)$,}
		\end{equation*}
		\begin{equation*}
		\int_{C_{2r}(T,a) \cap A}   \int_{\{z\} \times N(A,z)} f^{m} d\Haus{n-m-1} d\Haus{m}z \leq \delta,
		\end{equation*}
	\end{enumerate}
	then there exists a Borel set $  M \subseteq N(A)|\Clos C_{(11/8)r}(T,a) $ such that 
	\begin{equation*}
	\Haus{m}(\mathbf{p}[M])\geq \sigma r^{m},
	\end{equation*}
	\begin{equation*}
	\bm{\delta}_{A}(x + (r/2)u) = r/2 \quad \textrm{if $ (x,u) \in M $}, 
	\end{equation*}
	\begin{equation*}
	\Haus{n-m-1}\{v : (x,v) \in M\}    > 0 \quad \textrm{if $ x \in \mathbf{p}[M] $,}
	\end{equation*}
	\begin{equation*}
	\mathbf{p}[M] \subseteq A^{(m)}.
	\end{equation*}
\end{Lemma}

\begin{proof}
	We assume $ a =0 $ and we let $C_{t}=C_{t}(T,0)$ whenever $ 0 < t < \infty $.
	
	By \ref{localizing the Lusin property} we notice that $ N(\Clos (A \cap C_{4r})) $ satisfies the $ m $ dimensional Lusin (N) condition in $ C_{4r} $ and we replace $ A $ with $ \Clos (A \cap C_{4r}) $. We consider the diffeomorphism $ F : \Real{n} \rightarrow \Real{n} $ given by
	\begin{equation*}
	F(x)= x + (r/8)v -(4r)^{-1}|T_{\natural}(x)|^{2}v \quad \textrm{for $ x \in \Real{n} $}
	\end{equation*}
	and we compute  
	\begin{equation*}
	F^{-1}(x)= x - (r/8)v +(4r)^{-1}|T_{\natural}(x)|^{2}v,
	\end{equation*}
	\begin{equation*}
	\Der F(x)(u) = u - (2r)^{-1}(T_{\natural}(x) \bullet T_{\natural}(u)  )v,
	\end{equation*}
	\begin{equation*}
	\Der F^{-1}(x)(u) = u + (2r)^{-1}(T_{\natural}(x) \bullet T_{\natural}(u)  )v,
	\end{equation*}
	\begin{equation*}
	\Der^{2}F^{-1}(x)(u_{1}, u_{2}) = (2r)^{-1}(T_{\natural}(u_{1}) \bullet T_{\natural}(u_{2}) )v,
	\end{equation*}
	for $ x, u , u_{1}, u_{2} \in \Real{n}$. Moreover we notice
	\begin{equation}\label{Around a point where the set is almost flat in one direction:1}
	F^{-1}[\Clos C_{r}] \subseteq \Clos C_{(11/8)r},
	\end{equation} 
	\begin{equation}\label{Around a point where the set is almost flat in one direction:2}
	\sup_{x\in A}F(x)\bullet v \leq 3r/16, \quad \sup_{x \in A, |T_{\natural}(x)| \geq r} F(x)\bullet v \leq - r/16.
	\end{equation}
	Suppose $ L $ is the set of $(z,\eta) \in (F[A] \cap \Clos C_{r}) \times \mathbf{S}^{n-1} $ such that $(w-z) \bullet \eta \leq 0 $ whenever $ w \in F[A] $. We observe that $ L $ is compact and, noting \eqref{Around a point where the set is almost flat in one direction:1},
	\begin{equation}\label{Around a point where the set is almost flat in one direction:3}
	L \subseteq N(F[A])|F[\Clos C_{(11/8)r}].
	\end{equation}

	We define $ \nu_{F^{-1}} $ as in \ref{normal bundle and diffeomorphisms} and we prove that
	\begin{equation}\label{Around a point where the set is almost flat in one direction:4}
	\bm{\delta}_{A}(F^{-1}(z)+(r/2)(\mathbf{q}\circ \nu_{F^{-1}})(z,\eta)) = r/2 \quad \textrm{for $ (z,\eta) \in L $.}
	\end{equation}	
	In fact, if $ x= F^{-1}(z) $, $ \zeta = (\mathbf{q}\circ \nu_{F^{-1}})(z,\eta) $ and $ y \in A $, we compute
	\begin{flalign*}
	|\Der F(x)^{\ast}(\eta)|^{-1}\eta & =	(\Der F(x)^{\ast})^{-1}(\zeta) = \zeta + (2r)^{-1}(v \bullet \zeta)T_{\natural}(x),\\
	0 & \geq  (DF(x)^{\ast})^{-1}(\zeta) \bullet (F(y)-F(x)) \\
	&  = \zeta \bullet (y-x) + (4r)^{-1}(|T_{\natural}(x)|^{2}-|T_{\natural}(y)|^{2})(v \bullet \zeta)\\
	& \quad + (2r)^{-1}(T_{\natural}(x)\bullet (y-x))(v \bullet \zeta)\\
	& = \zeta \bullet (y-x) - (4r)^{-1}|T_{\natural}(y-x)|^{2}(v \bullet \zeta),\\
	|y-x-(r/2)\zeta|^{2} & = |y-x|^{2} + (r^{2}/4)- r(y-x) \bullet \zeta \\
	&  \geq |y-x|^{2} + (r^{2}/4)-(1/4)|T_{\natural}(y-x)|^{2}(v \bullet \zeta) \geq r^{2}/4.
	\end{flalign*}
	
	Let $ c_{0}= \Leb{n-1}(\mathbf{U}(0,1/(1+32^{2})^{1/2})) $ and we prove that
	\begin{equation}\label{Around a point where the set is almost flat in one direction:5}
	\Haus{n-1}(\mathbf{q}[L]) \geq  c_{0}.
	\end{equation}
	We consider the closed convex cone
	\begin{equation*}
	C = \{ (1-t) F(0)+ tx: x \in T \cap \mathbf{B}(0,4r), \; 0 \leq t < \infty  \}
	\end{equation*}
	and we notice that
	\begin{equation}\label{Around a point where the set is almost flat in one direction:6}
	\Dual \Nor(C, F(0)) = \Tan (C,F(0)) = \{ z - F(0) : z \in C  \},
	\end{equation}
	\begin{equation}\label{Around a point where the set is almost flat in one direction:7}
	\{ z : z \bullet v \leq 0 , \; |T_{\natural}(z)| \leq 4r \} \subseteq C.
	\end{equation}
	A direct computation shows that
	\begin{equation*}
	\mathbf{S}^{n-1} \cap	\Nor(C, F(0)) = \mathbf{S}^{n-1} \cap \{ \eta: 32/(1+32^{2})^{1/2} \leq \eta \bullet v \leq 1  \},
	\end{equation*}
	whence we readily infer that
	\begin{equation*}
	\Haus{n-1}(\Nor(C, F(0)) \cap \mathbf{S}^{n-1}) \geq c_{0}.
	\end{equation*}
	Therefore in order to prove \eqref{Around a point where the set is almost flat in one direction:5} it remains to check that
	\begin{equation*}
	\Nor(C, F(0)) \cap \mathbf{S}^{n-1} \subseteq \mathbf{q}[L].
	\end{equation*}
	Let $ \eta \in 	\Nor(C, F(0)) \cap \mathbf{S}^{n-1} $. In the case that $(z-F(0)) \bullet \eta \leq 0 $ for every $ z \in F(A) $, then it is obvious that $(F(0), \eta) \in L $ (notice that $ F(0) = (r/8)v \in C_{r} $). Therefore we assume that $ s = \sup \{ (z-F(0)) \bullet \eta : z \in F[A]  \} > 0 $. If $ z \in F[A] $ and $ (z-F(0)) \bullet \eta > 0 $ then we notice that
	\begin{equation*}
	z \notin C \;\; \textrm{by \eqref{Around a point where the set is almost flat in one direction:6}}, \qquad |T_{\natural}(z)| \leq 4r \;\; \textrm{by definition of $ F $,}
	\end{equation*}
	\begin{equation*}
	\textrm{$ z \bullet v > 0 $ by \eqref{Around a point where the set is almost flat in one direction:7}}, \qquad \textrm{$ |T_{\natural}(z)| < r $ and $|T^{\perp}_{\natural}(z)| < 3r/16 $ by \eqref{Around a point where the set is almost flat in one direction:2}},  
	\end{equation*}
	that means $ z \in C_{r} $. Therefore we select $ z_{0} \in F[A] $ such that $ (z_{0} - F(0)) \bullet \eta = s $ and, noting that the maximality of $ z_{0} $ implies that
	\begin{equation*}
	( w-z_{0}) \bullet \eta \leq 0 \quad \textrm{for every $ w \in F[A] $,}
	\end{equation*}
	we conclude that $ (z_{0},\eta) \in L $.
	
	We notice that $N(F[A])$ satisfies the $m$ dimensional Lusin (N) condition in $F[C_{4r}]$ by \ref{Lusin property and diffeomorphisms}. Therefore, employing \ref{Lusin condition and dimension of tangent space}, \ref{second fundamental form and diffeomorphisms}, \cite[4.8]{2017arXiv170801549S} and \eqref{Around a point where the set is almost flat in one direction:4} and noting that $ \Der F(F^{-1}(z))= \Der F^{-1}(z)^{-1} $ for $ z \in \Real{n} $, we infer at $\Haus{n-1}$ a.e.\ $ (z,\eta)\in L $ that 
	\begin{equation}\label{Around a point where the set is almost flat in one direction:8}
	\dim T_{F[A]}(z,\eta) = m,  \quad \Der F^{-1}(z)[T_{F[A]}(z,\eta)] = T_{A}(\nu_{F^{-1}}(z,\eta)),  
	\end{equation}
	\begin{equation}\label{Around a point where the set is almost flat in one direction:9}
	Q_{F[A]}(z,\eta)\geq 0, 
	\end{equation}
	\begin{equation}
	 Q_{A}(\nu_{F^{-1}}(z,\eta))(\tau, \tau) \geq - (r/2)^{-1}|\tau|^{2} \quad \textrm{for $ \tau \in T_{A}(\nu_{F^{-1}}(z,\eta)) $,}
	\end{equation}
	\begin{flalign}\label{Around a point where the set is almost flat in one direction:10}
	Q_{F[A]}(z,\eta) & = |(\Der F(F^{-1}(z)))^{\ast}(\eta)| Q_{A}(\nu_{F^{-1}}(z,\eta))\circ \textstyle \bigodot_{2}\big(\Der F^{-1}(z)| T_{F[A]}(z,\eta)\big) \nonumber\\
	&\quad  - (\Der^{2}F^{-1}(z)| \textstyle \bigodot_{2}T_{F[A]}(z,\eta)) \bullet \Der F(F^{-1}(z))^{\ast}(\eta).
	\end{flalign}
	In particular, by \cite[2.10.25]{MR0257325}, the same conclusion holds for $\Haus{m}$ a.e.\ $ z \in \mathbf{p}[L] $ and for $\Haus{n-m-1}$ a.e.\ $ \eta \in \{\zeta : (z,\zeta) \in L  \}$.
	We combine \ref{area formula for the gauss map} and the classical inequality relating the arithmetic and the geometric means of a family of non negative numbers (see \cite[pp.\ 29]{MR0274683}) to estimate
	\begin{flalign*}
	\Haus{n-1}(\mathbf{q}[L]) & \leq \int_{\mathbf{S}^{n-1}} \Haus{0}\{z : (z,\eta) \in L\}d\Haus{n-1}\eta  \\
	& = \int_{F[A]} \int_{\{z\} \times \{\eta : (z,\eta)\in L\} } \discr Q_{F[A]} d\Haus{n-m-1}d\Haus{m}z \\
	& \leq m^{-m} \int_{F[A]} \int_{\{z\} \times \{\eta : (z,\eta)\in L\} } (\trace Q_{F[A]})^{m}d \Haus{n-m-1}d\Haus{m}z. 
	\end{flalign*}
	We observe that if $ z \in \Clos C_{r} $, $ \eta \in \mathbf{S}^{n-1} $ and $ S \in \mathbf{G}(n,m) $, then
	\begin{equation*}
	\| \Der F(F^{-1}(z)) \| \leq 3/2, \quad  \| \Der F^{-1}(z) \| \leq 3/2,
	\end{equation*}
	\begin{equation}\label{Around a point where the set is almost flat in one direction:12}
	| \trace [(\Der^{2}F^{-1}(z)|\textstyle \bigodot_{2}S) \bullet \Der F(F^{-1}(z))^{\ast}(\eta)] | \leq (3/4)m r^{-1}.
	\end{equation}
	Therefore, noting \eqref{Around a point where the set is almost flat in one direction:8} and \eqref{Around a point where the set is almost flat in one direction:9}, we use \ref{Linear algebra trace} to infer that
	\begin{flalign*}
	&\trace\big[	Q_{A}(\nu_{F^{-1}}(z,\eta))\circ \textstyle \bigodot_{2}\big(\Der F^{-1}(z)| T_{F[A]}(z,\eta)\big)\big] \\
	&\quad \quad \leq c_{1}\big[ \trace\big( Q_{A}(\nu_{F^{-1}}(z,\eta))\big) + r^{-1}\big]
	\end{flalign*}
	whence we readily deduce using \eqref{Around a point where the set is almost flat in one direction:10} and \eqref{Around a point where the set is almost flat in one direction:12} that
	\begin{equation*}
	(\trace Q_{F[A]}(z,\eta))^{m} \leq c_{1}(f(\nu_{F^{-1}}(z,\eta))^{m} + r^{-m} )
	\end{equation*}
	for $ \Haus{n-1} $ a.e.\ $ (z,\eta) \in L $, where $ c_{1} $ is a constant depending only on $ m $.
	
	Since $ N(F(A))$ is countably $n-1$ rectifiable and $ L \subseteq N(F(A)) $ one may argue as in \cite[2.10.26]{MR0257325} to prove that $D =\{ z : \Haus{n-m-1}\{\zeta : (z,\zeta)\in L \} > 0 \}$ is a Borel subset of $\Real{n}$. Define
	\begin{equation*}
	M = \nu_{F^{-1}}[L|D],
	\end{equation*}
	and notice that it follows from \ref{normal bundle and diffeomorphisms} that $M$ is a Borel subset of $N(A)$ with
	\begin{equation*}
	\Haus{n-m-1}\{v : (x,v) \in M\}    > 0 \quad \textrm{for $ x \in \mathbf{p}[M] $.}
	\end{equation*}
	Moreover $\mathbf{p}[M]\subseteq \Clos C_{(11/8)r}$ by \eqref{Around a point where the set is almost flat in one direction:3} and $\bm{\delta}_{A}(x + (r/2)u) = r/2$ whenever $ (x,u) \in M $ by \eqref{Around a point where the set is almost flat in one direction:4}. Noting that $N(A,z)$ is contained in an $n-m$ dimensional plane whenever $ z \in A^{(m)} $ (see section \ref{section stratification}) and that
	\begin{equation*}
	\Lip(F|\Clos C_{(11/8)r}) \leq \sup_{z \in \Clos C_{(11/8)r}} \| \Der F(z) \| \leq 27/16,
	\end{equation*}
	we use \ref{normal bundle and diffeomorphisms} to estimate
\begin{flalign*}
&	\int \Haus{n-m-1} \{\zeta : (z,\zeta)\in L \} > 0 \}\, d\Haus{m}z \\
& \qquad \leq \int_{D \cap F(A)^{(m)}}\Haus{n-m-1}(N(A,z)) d\Haus{m}z \\
& \qquad \leq \Haus{n-m-1}(\mathbf{S}^{n-m-1})\Haus{m}[F(A^{(m)} \cap \mathbf{p}(M))] \\
& \qquad  \leq \Haus{n-m-1}(\mathbf{S}^{n-m-1})(27/16)^{m}\Haus{m}(\mathbf{p}(M)),
	\end{flalign*}
	\begin{flalign*}
&	\int_{F[A]}\int_{\{\zeta : (z,\zeta)\in L \}}  f(\nu_{F^{-1}}(z,\eta))^{m}d\Haus{n-m-1}\eta \, d\Haus{m}z \\
&	\qquad \leq 	c_{2}\int_{A}\int_{\{\zeta : (z,\zeta)\in M \}}f(z,\eta)^{m} d\Haus{n-m-1}\eta d\Haus{m}z,
	\end{flalign*}
	where $c_{2}$ is a constant depending on $m$ and $n$. Therefore,
	\begin{flalign}\label{Around a point where the set is almost flat in one direction:11}
	\Haus{n-1}(\mathbf{q}[L]) & \leq m^{-m} c_{1} c_{2}\int_{A}\int_{\{\zeta : (z,\zeta)\in M \}} f(z,\eta)^{m}d\Haus{n-m-1}\eta d\Haus{m}z \nonumber \\
	& \quad + m^{-m}c_{1} (27/16)^{m}\Haus{n-m-1}(\mathbf{S}^{n-m-1})\Haus{m}(\mathbf{p}[M])r^{-m}.
	\end{flalign}
	Noting \eqref{Around a point where the set is almost flat in one direction:5}, we choose $ \delta> 0 $ in \ref{Around a point where the set is almost flat in one direction:hp3} so that
	\begin{equation*}
	c_{0}-m^{-m}c_{1}c_{2}\delta \geq c_{0}/2
	\end{equation*}
	and we conclude from \eqref{Around a point where the set is almost flat in one direction:11} that
	\begin{equation*}
	\Haus{m}(\mathbf{p}[M]) \geq \sigma r^{m}
	\end{equation*}
	with $ \sigma = m^{m}c_{1}^{-1} (27/16)^{-m}\Haus{n-m-1}(\mathbf{S}^{n-m-1})^{-1}(c_{0}/2) $. Being $ \Haus{m}(\mathbf{p}[M]) > 0 $ and $ \Haus{n-m-1}(N(A,a)) > 0 $ for all $ a \in \mathbf{p}[M] $ it follows that $ \mathbf{p}[M] \subseteq A^{(m)} $.
\end{proof}

\begin{Theorem} \label{final result about second order rectifiability}
	Suppose $ 1 \leq m < n $ are integers, $ A \subseteq \Real{n} $ is a closed set, $\Omega \subseteq \Real{n} $ is an open set, $N(A)$ satisfies the $ m $ dimensional Lusin (N) condition in $ \Omega $, $ \Haus{m}(A \cap K)< \infty $ whenever $ K \subseteq \Omega $ is compact, for $ \Haus{m} $ a.e.\ $ a \in A \cap \Omega $ there exists $ v \in \mathbf{S}^{n-1} $ such that
	\begin{equation*}
	\lim_{r \to 0} r^{-1}\sup \{ v \bullet (x-a): x \in \mathbf{B}(a,r) \cap A \} =0, 
	\end{equation*}
	and there exists a non negative $\Haus{n-1}$ measurable function $f$ on $N(A)|\Omega$ such that
	\begin{equation*}
	\trace Q_{A}(a,u) \leq f(a,u) \quad \textrm{for $\Haus{n-1}$ a.e.\ $(a,u) \in N(A)|\Omega$,}
	\end{equation*}
	\begin{equation*}
	\int_{K \cap A}   \int_{\{z\} \times N(A,z)} f^{m} d\Haus{n-m-1} d\Haus{m}z < \infty,
	\end{equation*}
	whenever $ K \subseteq \Omega$ is compact.
	
	Then $ \Haus{m}(A \cap \Omega \sim A^{(m)})=0 $. In particular, $A \cap \Omega$ is countably $\rect{m}$ rectifiable of class $2$.
\end{Theorem}

\begin{proof}
	Firstly we notice that $C_{4r}(T,a) \subseteq \mathbf{B}(a, 4\sqrt{2}r) $ for every $ a \in \Real{n} $ and $ T \in \mathbf{G}(n, n-1) $. If $ \delta $ is given as in \ref{Around a point where the set is almost flat in one direction}, with the help of \cite[2.4.11]{MR0257325}, for $ \Haus{m} $ a.e.\ $ a \in A \cap \Omega $ we can select $ s > 0 $ and $ v \in \mathbf{S}^{n-1} $ such that $ \mathbf{B}(a, 4\sqrt{2}s) \subseteq \Omega $, 
	\begin{equation*}
	r^{-1}\sup \{ v \bullet (x-a): x \in \mathbf{B}(a,4\sqrt{2}r) \cap A \} \leq 1/16,
	\end{equation*}
	and
	\begin{equation*}
	\int_{C_{2r}(T,a) \cap A}   \int_{\{z\} \times N(A,z)} f^{m} d\Haus{n-m-1} d\Haus{m}z \leq \delta
	\end{equation*}
	for every $ 0 < r  < s $, where $ T \in \mathbf{G}(n, n-1) $ such that $T_{\natural}(v) =0 $. It follows from \ref{Around a point where the set is almost flat in one direction} that 
	\begin{equation*}
	\infHdensity{m}{A^{(m)}}{a} >0 \quad \textrm{for $ \Haus{m} $ a.e.\ $ a \in A \cap \Omega $.}
	\end{equation*}
	Since $ \Hdensity{m}{A^{(m)}}{a}=0 $ for $ \Haus{m} $ a.e.\ $ a \in  \Omega \sim A^{(m)} $ by \cite[2.10.19(4)]{MR0257325}, we infer that
	\begin{equation*}
	\Haus{m}(A \cap \Omega \sim A^{(m)}) =0.
	\end{equation*}
	The postscript follows from \cite[4.12]{2017arXiv170309561M}.
\end{proof}

\section{Proof of theorem \ref{second order rect for varifolds}}\label{section 4}

Here we prove Theorem \ref{second order rect for varifolds}. The main point will be to check that the closure\footnote{We take the closure in $ \Real{n} $ because Theorem \ref{final result about second order rectifiability} has been formulated for closed subsets in $ \Real{n} $.} in $ \Real{n} $ of the support $ S $ of $ V $ satisfies the hypothesis of the general criterion for $ \mathcal{C}^{2} $ rectifiability in \ref{final result about second order rectifiability}. These hypothesis have been already checked for $ V $ in several different papers, so we just need to collect them here.
\begin{enumerate}
	\item \emph{$ \Haus{m}(S \cap K) < \infty $ for every compact set $ K \subseteq \Omega$.} This follows combining the upper-semicontinuity of the density function $\bm{\Theta}^{m}(\|V\|, \cdot)$, see \cite[8.6]{MR0307015}, with the fact that $ \|V\| = \Haus{m}\restrict \bm{\Theta}^{m}(\|V\|, \cdot) $. In fact, we obtain the stronger conclusion $ \Haus{m} \restrict S \leq \theta^{-1} \|V\| $.
	\item \emph{$N(\Clos S)$ satisfies the $m$ dimensional Lusin (N) condition and
		\begin{equation*}
		\trace Q_{\Clos S}(a,u) \leq h \quad \textrm{for $ \Haus{n-1} $ a.e.\ $(a,u) \in N(\Clos S)|\Omega $.}
		\end{equation*}}
Noting \cite[2.8]{MR3466806}, this is a special case of \cite[3.7]{2019arXiv190310379S}.
\item \emph{For $ \|V\| $ a.e.\ $ a \in \Omega $ there exists an $ m $ dimensional plane $ T $ such that
\begin{equation*}
\lim_{r \to 0}r^{-1}\sup\{ \mathbf{\delta}_{T}(x-a) : x \in \mathbf{B}(a,r) \cap S \} =0.
\end{equation*}
} This follows from \cite[17.11]{MR756417}.
\end{enumerate}


\medskip 

\noindent Institut f\"ur Mathematik, Universit\"at Augsburg, \newline Universit\"atsstr.\ 14, 86159, Augsburg, Germany,
\newline mario.santilli@math.uni-augsburg.de

\end{document}